\pdfoutput=1
\documentclass[a4paper,
12pt,
reqno,
]{amsart}

\usepackage[utf8]{inputenx}
\usepackage{amsmath, amssymb, mathtools, graphics, fullpage}
\usepackage[colorlinks]{hyperref}
\hypersetup{citecolor=blue}
\usepackage[margin=1.25in]{geometry}
\usepackage{subfig}
\usepackage{graphicx}

\usepackage{nicefrac}
\let\ofrac=\frac
\let\nfrac=\nicefrac

\newcommand{\qede}{\hfill $\diamond$}
\newcommand{\M}{\mathbf{M}}
\newcommand{\Z}{\mathbb{Z}}
\newcommand{\A}{\mathcal{A}}
\newcommand{\K}{\mathcal{K}}
\newcommand{\kl}{\mathcal{K}\mathcal{L}}
\newcommand{\R}{\mathbb{R}}

\renewcommand{\H}{\mathbb{H}}
\renewcommand{\S}{\mathcal{S}}
\newcommand{\eps}{\varepsilon}
\newcommand{\e}{\mathbf{e}}
\renewcommand{\c}{\mathbf{c}}
\newcommand{\p}{\mathbf{p}}
\newcommand{\x}{\mathbf{x}}
\newcommand{\y}{\mathbf{y}}
\newcommand{\w}{\mathbf{w}}
\newcommand{\z}{\mathbf{z}}
\newcommand{\n}{\mathbf{n}}

\theoremstyle{plain}
\newtheorem{theorem}{Theorem}[section]
\newtheorem{corollary}[theorem]{Corollary}
\newtheorem{lemma}[theorem]{Lemma}
\newtheorem{proposition}[theorem]{Proposition}

\theoremstyle{definition}
\newtheorem{definition}[theorem]{Definition}
\newtheorem{remark}[theorem]{Remark}

\begin{document}
\title[Lyapunov Functions and Non-Convex Douglas--Rachford]{A Lyapunov
  Function Construction for a Non-Convex Douglas--Rachford Iteration}
\author{Ohad Giladi \and Björn S.\ Rüffer}

\address{School of Mathematical and Physical Sciences, University of Newcastle,\linebreak[4] Callaghan, NSW 2308, Australia}
\email{ohad.giladi@newcastle.edu.au, bjorn.ruffer@newcastle.edu.au}

\begin{abstract} While global convergence of the Douglas--Rachford
  iteration is often observed in applications, proving it is still
  limited to convex and a handful of other special cases.  Lyapunov
  functions for difference inclusions provide not only global or local
  convergence certificates, but also imply robust stability, which means
  that the convergence is still guaranteed in the presence of persistent
  disturbances.  In this work, a global Lyapunov function is constructed
  by combining known local Lyapunov functions for simpler, local
  sub-problems via an explicit formula that depends on the problem
  parameters. Specifically, we consider the scenario where one set
  consists of the union of two lines and the other set is a line, so
  that the two sets intersect in two distinct points. Locally, near each
  intersection point, the problem reduces to the intersection of just
  two lines, but globally the geometry is non-convex and the
  Douglas--Rachford operator multi-valued.  Our approach is intended to
  be prototypical for addressing the convergence analysis of the
  Douglas--Rachford iteration in more complex geometries that can be
  approximated by polygonal sets through the combination of local,
  simple Lyapunov functions.
\end{abstract}


\maketitle

\section{Introduction}\label{sec intro}

The Douglas--Rachford iteration was originally introduced
in~\cite{DR56}, subsequently generalized in \cite{LM79}, and is a well
known method for finding a point in the intersection of two or more
closed sets in a Hilbert space. It has found various applications both
in the case when the involved sets are convex \cite{BCL02,LM79} and in
the case when at least one set is
non-convex~\cite{ABT14,ERT07,GE08}. The convergence of the algorithm
in the latter case is still not fully understood to date.

When both sets are convex, it is known that the Douglas--Rachford
operator is firmly non-expansive (e.g., via \cite[Theorem 12.2]{GK90}
and repeated application of \cite[Theorem 12.1]{GK90}) and, if the
operator has a fixed point, the algorithm converges if the ambient
space is finite dimensional and converges weakly if the space is
infinite dimensional~\cite{Opi67}. Despite its use in applications,
much less is substantiated theoretically about the convergence
behavior of the algorithm if the sets are non-convex. Specific cases
where convergence proofs have been obtained include the case where the
first set is the unit Euclidean sphere and the second set is a
line~\cite{BS11,AB13,Ben15}.  In~\cite{BS11}, the authors show local
convergence to the intersection points, that in some special cases one
can obtain global convergence, and that in the non-feasible case the
Douglas--Rachford iteration is divergent. A stronger result with a
larger, explicit domain of convergence was proved
in~\cite{AB13}. In~\cite{Ben15}, the author gives an explicit
construction of a Lyapunov function for the Douglas--Rachford
iteration in the case of the sphere and a line. This in turn implies
global convergence for all points which are not in the subspace of
symmetry. In fact, the construction in~\cite{Ben15} can be used to
prove a type of convergence which is stronger than norm
convergence~\cite{Gil16}, cf.\ Section~\ref{sec:role-lyap-funct}.  In
\cite{ABT16} the global behavior of the Douglas--Rachford iteration
for the intersection of a half-space with a possibly non-convex set is
analyzed, which has applications in combinatorial optimization
problems. Global convergence in the case that one of the involved sets
is finite is shown in~\cite{MR3628318}.  The authors of~\cite{BLSSS16}
study ellipses and $p$-spheres intersecting a line, while proving
local convergence and employing computer-assisted graphical methods
for further analysis. The authors of \cite{DT17} employ the
Douglas--Rachford iteration for finding a zero of a function and
generalize the Lyapunov function construction of~\cite{Ben15} using
concepts from nonsmooth analysis to this case. In~\cite{HL13}, the
authors show local linear convergence of non-convex instances of the
Douglas--Rachford iteration based on the notion of local subfirm
nonexpansiveness and coercivity conditions. A generalization to
superregular (possibly non-convex) sets is given in~\cite{MR3438115},
as well as in the forthcoming
works~\cite{dao16:_linear_conver_projec_algor,dao17:_dougl_rachf}.
In~\cite{LP16} the Douglas--Rachford method is employed to minimize
the sum of a proper closed function and a smooth function with a
Lipschitz continuous gradient, showing convergence when the functions
are semi-algebraic and the step-size is sufficiently small.
In~\cite{MR3227480} local convergence is shown in the case of
non-convex sets that are finite unions of convex sets, while it is
also demonstrated that convergence may fail in the case of more
general sets.  Difference inclusions more general than the
Douglas--Rachford iteration along with their convergence behavior have
also been studied in the systems theory literature~\cite{KT05}.

In this paper we prove global convergence of a Douglas--Rachford
iteration (in fact, we even prove global robust $\kl$-stability) for
yet another specific case of a non-convex set, consisting of two
non-parallel lines, and a second set, which is also a line, so that
the first and the second set intersect in exactly two points, cf.\
Fig.~\ref{figThreeLines}.
\begin{figure}[htb!]  \centering \includegraphics{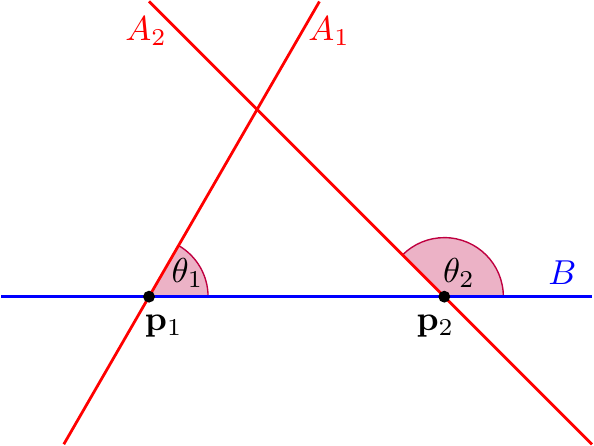}
  \caption{The geometry studied in the paper: One non-convex set,
    consisting of two lines (red), and another convex set, consisting of
    one line (blue), with two unique intersection points $\p_{1}$ and
    $\p_{2}$.}
  \label{figThreeLines}
\end{figure} This scenario is intended to be prototypical for the
study of the intersection of polygonal sets, which could be
approximations of norm-spheres or ellipses. We remark that local
convergence for the scenario is already proven
in~\cite{MR3227480}. And\label{dt17-discussion} while~\cite{DT17} may
seem to cover the current scenario as a special case, it does in fact
not, as the non-convex set $A$ in this scenario cannot be represented
as the graph of a function defined on the set $B$, and because for a
global Lyapunov function construction the two isolated attractors
$\p_{1}$ and $\p_{2}$ are in conflict with a convexity assumption on
the energy function in~\cite{DT17} (both $\p_{i}$ would have to be
local minima).

Our main contribution is the construction of a global Lyapunov
function for robust $\kl$-stability in Theorem~\ref{main thm}. Unlike
previous contributions to the Douglas--Rachford convergence analysis
based on Lyapunov functions, our construction follows the
divide-and-conquer paradigm of re-using known, \emph{local} Lyapunov
functions in our construction of a \emph{global} Lyapunov function.
To this end we use the global Lyapunov functions for the
Douglas--Rachford iterations corresponding to the intersection problem
between two lines.  These are essentially the distance to the
intersection point and, in the problem depicted in
Fig.~\ref{figThreeLines}, they are local Lyapunov functions near each
intersection point for the non-convex Douglas--Rachford difference
inclusion.  The novelty in our contribution is that for a range of
problem parameters we can then combine these local Lyapunov functions
into a global Lyapunov function for the Douglas--Rachford iteration
for the non-convex problem.  We can thus deduce that the
Douglas--Rachford iteration converges in a non-convex scenario
provided certain conditions on the geometry are met. At the same time,
we know from extensive numerical experiments that Douglas--Rachford
iterations are not guaranteed to converge to a feasible point for all
problem parameters. One of the advantages of using Lyapunov functions
is the fact that in many cases the existence and properties of a
Lyapunov function imply types of convergence which are stronger than
point-wise norm convergence. In
Corollary~\ref{cor:robust-kl-convergence-rate} explicit error bounds
for the solutions of perturbed versions of the Douglas--Rachford
iteration are given for certain choices of angles $\theta_1$,
$\theta_2$. This means that even if we allow small perturbations of
the elements of the Douglas--Rachford iterates, we still obtain
\emph{uniform} convergence on bounded sets.

This paper is organized as follows.  Notation is introduced in
Section~\ref{sec:notation}. In Section~\ref{sec:role-lyap-funct} we
review known concepts and results from~\cite{KT05} on robust stability
with respect to two measures for difference inclusions. This includes
the definition of a Lyapunov function as it is commonly used in the
modern systems and control literature. The definition of the
Douglas--Rachford iteration is recalled in
Section~\ref{sec:dougl-rachf-algor}.  In
Section~\ref{sec:comb-local-lyap} we discuss the Douglas--Rachford
iteration for two sets $A$ and $B$ with $A$ consisting of two
non-parallel lines and $B$ another line, so that the intersection of
$A$ and $B$ consists of two points. A technical proof in this section
has been postponed to Appendix~\ref{app:proof-prop}. In
Section~\ref{sec:persp-open-probl} several open problems for further
research are formulated, while Section~\ref{sec:conclusions} concludes
the paper.

Sagemath~\cite{sagemath} code in the form of a jupyter notebook is
available at~\cite{GR17} for the interested reader to experiment
with. This code implements the geometry described above as well as the
associated Douglas--Rachford operator.

\section{Notation}
\label{sec:notation}

Denote $\R_+ \coloneqq [0,\infty[$ and $\Z_+ \coloneqq
\{0,1,2,\dots\}$. By bold letters, like $\e$, $\x$, and $\p$, we
denote vectors in the Hilbert space $\H$, and for most of this paper
$\H$ is either $\R^{d}$ or $\R^{2}$.  Let $B(\x,\rho)$ ($B[\x,\rho]$)
denote the open (closed) ball in $\R^{d}$ centered at $\x$ with radius
$\rho$, with respect to the Euclidean norm.  Let $\e_1,\e_2$ denote
the standard basis vectors in $\R^2$, $\e_1 \coloneqq (1,0)$, $\e_2
\coloneqq (0,1)$.  By $\M_{\theta}$ we denote the rotation
matrix $$\begin{bmatrix} \cos\theta & \sin \theta \\ -\sin\theta~&
  \cos\theta
\end{bmatrix}$$ acting on $\R^{2}$.  A function $\beta\colon
\R_+\times \R_+\to \R_+$ is said to be of class $\kl$ if for every
$t>0$, $\beta(\cdot,t)$ is continuous, strictly increasing, and
$\beta(0,t) = 0$, and also for every $s\in \R_+$, $\beta(s,\cdot)$ is
decreasing, and satisfies $\beta(s,t) \stackrel{t\to
  \infty}{\longrightarrow} 0$.  A function $\varphi\colon \R_+\to \R_+$
is said to be of class $\K_\infty$ if it is continuous, strictly
increasing, unbounded, and satisfies $\varphi(0) = 0$.

\section{The Role of Lyapunov Functions in Robust Stability}
\label{sec:role-lyap-funct}

In this section we take a detour to recall some known definitions and
results from the theory of discrete time dynamical systems and
difference inclusions. For more information on the subject the
interested reader is referred to~\cite{KT05}.  This will serve as a
basis for the following section where we will demonstrate that certain
instances of the Douglas--Rachford iteration are robustly
$\kl$-stable.

Let $U\subset \R^{d}$, $T\colon U \rightrightarrows U$ be a
multi-valued map, and consider the difference inclusion
\begin{align}\label{diff incl} \x_{n+1} \in T\x_n, \quad n\in \Z_+.
\end{align} A \emph{solution} to the initial value problem given by
the difference inclusion~\eqref{diff incl} with initial condition
$\x_0\in U$, which we denote by $$\phi(\x_{0},\cdot)\colon \Z_+ \to
\R^d,$$ is a function that satisfies $\phi(\x_{0},0) = \x_{0}$ and
$$\phi(\x_{0},n+1) \in T\big(\phi(\x_{0},n)\big)$$ for all
$n\in\Z_{+}$. Note that for difference inclusions there may well be
more than one solution for the same initial value problem.  The set of
all solutions to~\eqref{diff incl} is denoted by $\S(\x_{0},T)$. We
will also commonly speak of solutions of the difference
inclusion~\eqref{diff incl} and really mean solutions to a
corresponding initial value problem that will be clear from the
context. A \emph{periodic solution}
$\phi(\x_{0},\cdot)\colon\Z_{+}\to\R^{d}$ is a solution of~\eqref{diff
  incl} that is periodic in $n$, i.e., there exists a $K\in\Z_{+}$,
$K>1$, such that $\phi(\x_{0},n+K)=\phi(\x_{0},n)$ for all
$n\in\Z_{+}$.  A \emph{periodic orbit} is the image of a periodic
solution in $\R^{d}$, i.e., the set $\{\phi(\x_{0},n)\colon
n\in\Z_{+}\}$.  We define several stability properties for the
difference inclusion~\eqref{diff incl}.
\begin{definition}[$\kl$-stability] Let $\omega_1,\omega_2\colon
  \R^d\to \R_+$ be continuous functions. The difference
  inclusion~\eqref{diff incl} is said to be \emph{$\kl$-stable with
    respect to $(\omega_1,\omega_2)$} iff there exists $\beta \in \kl$
  such that for every $\x\in \R^d$, every $\phi\in \S(\x,T)$ and every
  $n\in \Z_+$,
  \begin{align*} \omega_1(\phi(\x,n)) \le \beta(\omega_2(\x),n).
  \end{align*}
\end{definition} For example, if $\omega_{1}=\omega_{2}$ is just the
distance to some set of interest, then $\kl$-stability says that
solutions for any initial conditions will converge to this set with a
uniform rate of convergence (which is encoded in $\beta$).  To define
the stronger notion of \emph{robust $\kl$-stability} we need to
introduce a few additional concepts.  Let $\sigma\colon \R^d\to \R_+$
and for a set $K\subseteq \R^d$ define its dilation with respect to
$\sigma$ by
\begin{align*} K_\sigma \coloneqq \bigcup_{\x\in K}B[\x,\sigma(\x)].
\end{align*} Given a map $T\colon \R^d\rightrightarrows \R^d$, define
the $\sigma$-perturbation of $T$ by
\begin{align*} T_\sigma \x \coloneqq \bigcup_{\y \in
  T(B[\x,\sigma(\x)])}B[\y,\sigma(\y)],
\end{align*} and let $\S_\sigma(\x_{0},T) \coloneqq
\S(\x_{0},T_\sigma)$ be the collection of all solutions to the
perturbed difference inclusion $\x_{n+1} \in T_\sigma \x_n$ with
initial condition $\x_0$. Notice that if $\sigma \equiv 0$, the
constant zero function, then $T_\sigma =T$.  Finally, given a
continuous function $\omega_1 \colon \R^d \to \R_+$, define the two
sets
\begin{align} \A_\sigma \coloneqq \A_\sigma(T,\omega_{1}) & \coloneqq
                                                            \Big\{\x \in \R^d\colon \sup_{n\in \Z_+}\sup_{\phi\in \S_\sigma(\x,
                                                            T)}\omega_1(\phi(\x,n)) = 0\Big\}\nonumber\\ \intertext{and}
  \label{def A} \A \coloneqq \A(T,\omega_{1}) & \coloneqq \Big\{ \x\in
                                                \R^d\colon \sup_{n\in \Z_+}\sup_{\phi\in \S(\x,
                                                T)}\omega_1(\phi(\x,n)) = 0\Big\}.
\end{align}
\begin{definition}[Robust $\kl$-stability]\label{def rob kl} Let
  $\omega_1,\omega_2 \colon \R^d \to \R_+$ be continuous. The difference
  inclusion~\eqref{diff incl} is said to be \emph{robustly $\kl$-stable
    with respect to $(\omega_1,\omega_2)$} iff there exists a continuous
  function $\sigma\colon \R^d\to \R_+$ such that
  \begin{enumerate}
  \item for all $\x\in \R^d\setminus \A$, $\sigma(\x)>0$;
  \item $\A_\sigma = \A$;
  \item the difference inclusion $\x_{n+1} \in T_\sigma \x_n$ is
    $\kl$-stable with respect to $(\omega_1,\omega_2)$.
  \end{enumerate}
\end{definition}
\begin{definition}[Lyapunov function]\label{def lyap} Let
  $\omega_1,\omega_2\colon \R^d \to \R_+$ be two continuous functions. A
  function $V\colon \R^d\to \R_+$ is said to be a \emph{Lyapunov
    function with respect to $(\omega_1,\omega_2)$} for the difference
  inclusion~\eqref{diff incl} iff there exist $\varphi_1,\varphi_2 \in
  \K_\infty$ and $\gamma \in [0,1)$ such that for all $\x\in \R^d$,
  \begin{eqnarray}
    \label{cond omegas} & \varphi_1(\omega_1(\x)) \le V(\x) \le
                          \varphi_2(\omega_2(\x));& \\
    \label{cond decay} & \sup_{\y \in T\x}V(\y) \le \gamma V(\x);& \\
    \label{cond invar set} & V(\x) = 0 \iff \x \in \A, &
  \end{eqnarray} where $\A$ is defined as in~\eqref{def A}.
\end{definition} There is an intimate connection between the stability
properties of the difference inclusion~\eqref{diff incl} and the
existence and properties of associated Lyapunov functions. In
particular, the following is known.
\begin{theorem}[Theorem 2.8 in~\cite{KT05}]\label{thm KT} Assume that
  $T\colon \R^d\rightrightarrows \R^d$ is such that $T\x$ is compact for
  all $\x\in \R^d$. Suppose also that there exists a continuous Lyapunov
  function on $\R^d$ with respect to two continuous functions
  $\omega_1$, $\omega_2$. Then the difference inclusion~\eqref{diff
    incl} is robustly $\kl$-stable with respect to $(\omega_1,\omega_2)$.
\end{theorem}

\section{The Douglas--Rachford Iteration}
\label{sec:dougl-rachf-algor}

For two sets the algorithm can be described as follows. Given a
non-empty set $A$ in a Hilbert space $(\H,\|\cdot\|)$ and a point
$\x\in \H$ denote $d(\x,A) \coloneqq \inf_{\z \in A}\|\x-\z\|$. The
projection operator $P_A\colon \H \rightrightarrows \H$ is given by
\begin{equation*} P_A \x \coloneqq \left\{\y \in \H\colon \|\x-\y\| =
    d(\x,A)\right\}
\end{equation*} and in general it can be multi-valued, but is
single-valued if $A$ is non-empty, closed, and convex. Given two
closed, non-empty sets $A,B \subseteq \H$, define the
Douglas--Rachford operator $T_{A,B}\colon \H\rightrightarrows \H$ by
\begin{equation*} T_{A,B} \coloneqq \frac{I+R_BR_A}{2},
\end{equation*} where $I\colon \H\to \H$ is the identity operator,
and, given a set $A\subseteq \H$, $R_A$ is the reflection operator
given by $R_A \coloneqq 2P_A - I$. The case where $A\cap B \neq
\emptyset$ is known as the feasible case. In this paper we will only
discuss the feasible case. Specifically, we consider the convergence
behavior of the difference inclusion $\x_{n+1} \in T_{A,B}\x_n$, with
$n\in\Z_{+}$, and $\x_{0}\in \H$, which is known as a
Douglas--Rachford iteration of $\x_{0}$.

\section{Combining Local Lyapunov Functions to Global Ones}
\label{sec:comb-local-lyap}

Finding Lyapunov functions is in general the hard part of Lyapunov
stability analysis. A divide-and-conquer inspired approach is to try
and find Lyapunov functions for simple sub-problems and then combine
these into a Lyapunov function for the general case of interest.

In this section we demonstrate that this can be done in a prototypical
non-convex scenario for the Douglas--Rachford iteration. In this
scenario it is very easy to formulate global Lyapunov functions for
the sub-problems (intersections of two lines). In the general version
of the problem these functions become then local Lyapunov functions,
i.e., they satisfy the descent condition~\eqref{cond decay} only
locally, that is, sufficiently close to a fixed point of the
difference inclusion~\eqref{diff incl}. The challenge is then to
combine these local Lyapunov functions to one global Lyapunov
function, which we demonstrate in Theorem~\ref{main thm}.

There are several motivations to consider the very simple geometry in
this section. Firstly, the case considered here is possibly the
simplest non-convex geometry where the Douglas--Rachford iteration
converges globally to feasible points for a range of problem
parameters.  Secondly, it is the first case known to the authors where
a global Lyapunov function for the Douglas--Rachford iteration has
been constructed from simpler, known local Lyapunov
functions. Thirdly, via approximation of circles, ellipses or function
graphs through polygons it is not unreasonable to expect that a
refined and possibly more localized version of our method could also
provide (alternative) Lyapunov function constructions for more
involved non-convex geometries like circle and line
(cf.~\cite{Ben15}), ellipse and line (cf.~\cite{BLSSS16}), or general
function graphs and line (cf.~\cite{DT17}). This could open the door
to novel sufficient conditions for the convergence of
Douglas--Rachford iterations in non-convex scenarios.

\subsection{Douglas--Rachford Iteration for Two Intersecting Lines}
\label{sec:dougl-rachf-iter-1+1}

A case which is elementary and well understood is the case of two
straight lines in $\R^2$, cf.\ Fig.~\ref{figTwoLines}. For a more
general treatment of the intersection of two subspaces we refer the
reader to \cite{MR3233066}.
\begin{figure}[htb!]  \centering \includegraphics{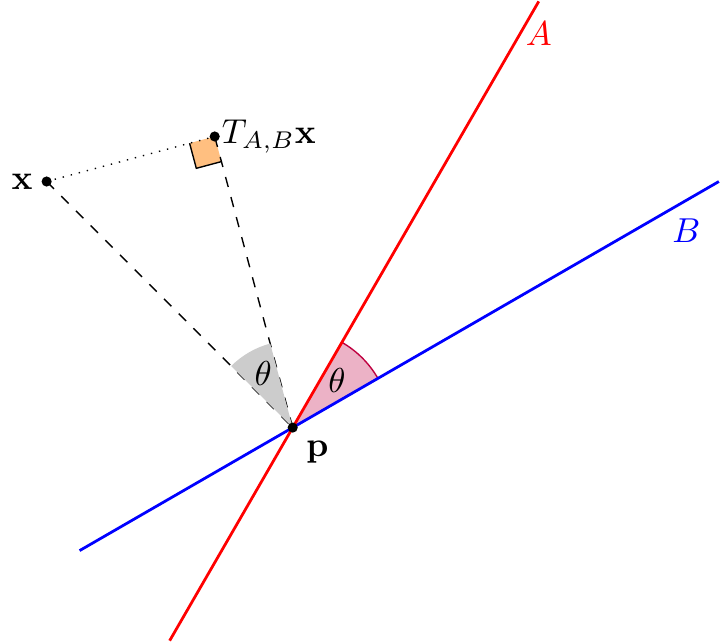}
  \caption{A Douglas--Rachford step for the case of two lines in the
    plane. Notice that the triangle $\triangle (T_{A,B}\x /\p / \x)$ is a
    right triangle.}
  \label{figTwoLines}
\end{figure} The qualitative behavior for this scenario, here
presented from a Lyapunov function perspective, is summarized as
follows.

\begin{proposition}\label{prop decay V's} Suppose that $A$ and $B$ are
  two non-parallel straight lines in $\R^2$ which intersect at a point
  $\p$, cf.\ Fig.~\ref{figTwoLines}. For simplicity we assume $B$ is the
  $x$-axis. Assume also that the angle from $B$ to $A$ is $\theta \in
  ]0,\pi[$.  Then the Douglas--Rachford operator $T_{A,B}$ is
  single-valued, affine, and is given by
  \begin{align} T_{A,B}\,\x = \p + \cos\theta \M_{\theta} (\x-\p),
    \label{eq:4}
  \end{align} for all $\x\in\R^{2}$.  Moreover, the function $V\colon
  \R^2\to [0,\infty)$ given by
  \begin{align} V(\x) \coloneqq \|\x-\p\|^2,
    \label{eq:3}
  \end{align} satisfies \eqref{cond omegas} with
  $\varphi_{1}(r)=\varphi_{2}(r)=r^{2}$ and
  $\omega_{1}(\x)=\omega_{2}(\x)=\|\x-\p\|$,
  \begin{align} V(T_{A,B}\x) =( \cos^2\theta) V(\x) < V(\x)\quad
    \text{whenever}\quad \x\ne \p\,,
    \label{eq:2}
  \end{align} as well as $V(\x)=0$ if and only if $\x=\p$.  That is,
  $V$ is a global Lyapunov function for the Douglas--Rachford iteration,
  and hence the latter is robustly $\mathcal{KL}$-stable.
\end{proposition}

\begin{proof} The result follows from the more general case
  in~\cite[Theorem~4.1 and Section~5]{MR3233066}, up to translation by
  $\p$, as well as an application of Theorem~\ref{thm KT}.
\end{proof} The existence of a global Lyapunov function guarantees
that the Douglas--Rachford iteration converges to a fixed point for
any initial condition.

Notice that we chose $V$ to be the \emph{squared} distance to $\p$. In
the control theory literature the choice of a quadratic Lyapunov
function is a de facto standard for two reasons. One reason is the
chain rule that simplifies the computation of
$\frac{d}{dt}V(x(t))=\nabla V(x) f(x)$ in continuous-time dynamics
$\dot x=f(x)$ at a point $x=x(t)$ without a requirement to compute
solutions to a differential equation. The other reason is the added
smoothness (compared to simply using the distance) at the reference
point, which is beneficial in robustness analysis (especially in
continuous-time systems, cf.~\cite{Kel15}) and control design (see,
e.g., \cite{sontag1998-mathematical-control-theory}).

\subsection{Douglas--Rachford Iteration for Two Lines Intersecting
  with a Third Line}
\label{sec:dougl-rachf-iter-2+1}

Now we assume that $A_1$, $A_2$, are two non-parallel straight lines
that each form a positive angle with the positive $x$-axis, and let
$A$ be given by $A \coloneqq A_1 \cup A_2$.  We assume that $B$ is the
$x$-axis, and that we have $A_1\cap B \eqqcolon \{\p_1\}$, $A_2\cap B
\eqqcolon \{\p_2\}$.

\label{single-point-intersection} The case when $\p_{1}=\p_{2}$, i.e.,
all three lines intersect in a single point, is not very different
from the discussion in the previous subsection. In fact, it can be
shown that the (squared) distance to the common intersection point is
a global Lyapunov function for the Douglas--Rachford iteration.

Here we concentrate on the more interesting case when
$\p_{1}\ne\p_{2}$. Without loss of generality we may assume $\p_1 =
-\nfrac{1}{2}\,\e_1$, $\p_2 = \nfrac{1}{2}\, \e_1$.  Denote by
$\theta_1,\theta_2$, the angles of $A_1,A_2$, respectively, with the
positive $x$-axis, as denoted in Fig.~\ref{figThreeLines}.  We assume
from here onwards that $0 < \theta_1 \le \pi/2$ and $\theta_1 <
\theta_2 < \pi$ (we exclude the cases $\theta_1 = 0$, $\theta_2 = 0$,
and $\theta_1 = \theta_2$ since in this case we have parallel lines,
or lines that coincide).  The following three sets in $\R^2$ are of
further interest,
\begin{align}\label{def D's} \nonumber D_1 & \coloneqq \big\{\x\in
                                             \R^2\colon d(\x,A_1) < d(\x,A_2)\big\}, \\ D_2 & \coloneqq \big\{\x\in
                                                                                              \R^2\colon d(\x,A_1) > d(\x,A_2)\big\}, \text{ and} \\ \nonumber D_3 &
                                                                                                                                                                     \coloneqq \big\{\x\in \R^2\colon d(\x,A_1) = d(\x,A_2)\big\},
\end{align} see Fig.~\ref{figDomains}.  It is these sets that
determine whether $T_{A,B}$ is multi-valued or singleton-valued, i.e.,
\begin{align}\label{def T} T_{A,B}\,\x = \begin{cases} T_{A_1,B}\,\x &
    \text{when }\x \in D_1, \\ T_{A_2,B}\,\x & \text{when }\x \in D_2, \\
    \big\{T_{A_1,B}\,\x,T_{A_2,B}\,\x\big\} & \text{when }\x\in
    D_3.\end{cases}
\end{align}
\begin{figure}[htb!]  \centering \includegraphics{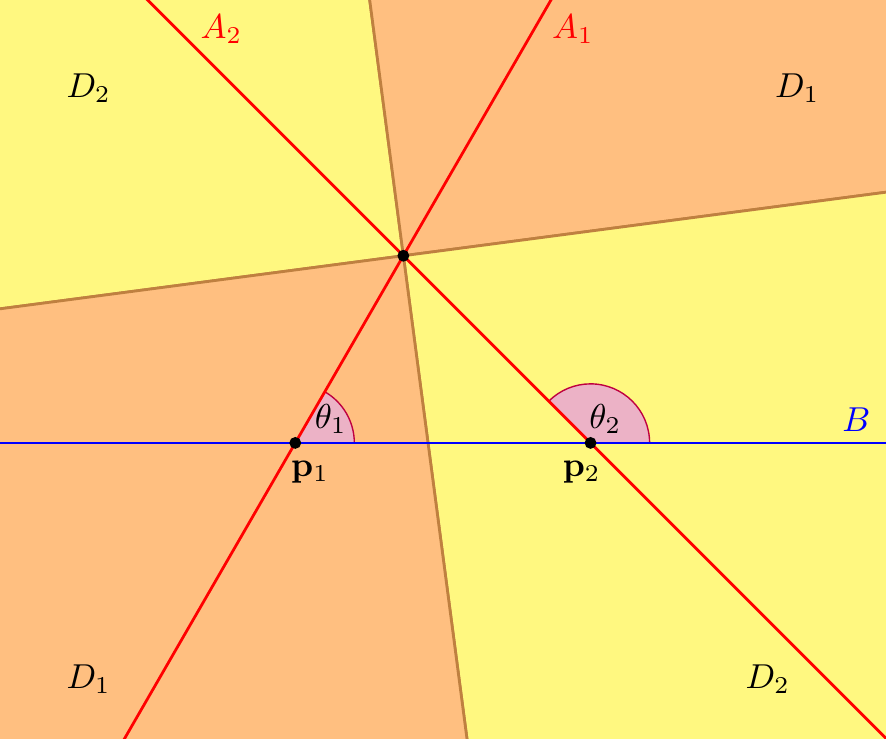}
  \caption{Regions where the Douglas--Rachford operator is single,
    respectively, multi-valued in the case of two lines (red) and one line
    (blue). The orange domain is $D_1$, the yellow domain is $D_2$ (the
    operator is singleton valued in both cases), and the two brown lines
    are $D_3$ (here the operator has two values).}
  \label{figDomains}
\end{figure} For $i=1,2$, let $V_i\colon \R^2 \to \R_+$ be the
functions defined by
\begin{align}\label{def V's} V_i(\x) \coloneqq \|\x-\p_i\|^2.
\end{align} Following the reasoning of Proposition~\ref{prop decay
  V's}, these are now \emph{local} Lyapunov functions for the
Douglas--Rachford iteration
\begin{align} \x^{+}\in T_{A,B}\,\x,
  \label{eq:1}
\end{align} i.e., if $\x_{0}$ is already sufficiently close to a fixed
point $\p_{i}$ then the corresponding sub-level set of $V_{i}$,
$\{\x\in\R^{2}: V_{i}(\x)\leq V_{i}(\x_{0})\}$, is completely
contained in $D_{i}$ and hence invariant under~\eqref{eq:1}. By the
decay condition~\eqref{eq:2} the sequence generated by~\eqref{eq:1}
must converge to $\p_{i}$.

However, if $\|\x_{0}-\p_{i}\|$ is too large for the sub-level set to
be completely contained in $D_{i}$, then it is \emph{a priori} not
clear to which point solutions of~\eqref{eq:1} emerging from $\x_{0}$
converge, or whether they converge at all.

Theorem~\ref{main thm} establishes that the globally defined, local
Lyapunov functions can indeed be combined to a global Lyapunov
function
\begin{align} V(\x)\coloneqq f\big(V_{1}(\x),V_{2}(\x)\big),
\end{align} provided a sufficient condition on the angles $\theta_{1}$
and $\theta_{2}$ is met. It is common in Lyapunov stability analysis
that conditions are only sufficient and not necessary (see
\cite{Kel15} on the concept of converse Lyapunov functions; their
existence proofs are usually non-constructive). This global Lyapunov
function in turn is a certificate for the global asymptotic stability
of the set $\{\p_{1},\p_{2}\}$ of fixed points for the iterative
scheme~\eqref{eq:1}, that is, this Lyapunov function establishes among
other properties that \emph{every} solution of~\eqref{eq:1} converges
either to $\p_{1}$ or to $\p_{2}$ for certain configurations of angles
$\theta_{1}$ and $\theta_{2}$.

Before we can derive Theorem~\ref{main thm}, we need to establish a
number of technical results that are summarized in the following
proposition.

\begin{proposition}
  \label{prop:thm-a-helper} Given $\rho>0$, let
  \begin{align*} \mathcal{B}_1(\rho) & \coloneqq \big\{\x \in
                                       \R^2\colon V_1(T_{A_2,B}\x) > \rho V_1(\x)\big\} \\
    \mathcal{B}_2(\rho) & \coloneqq \big\{\x\in\R^2\colon V_2(T_{A_1,B}\x)
                          > \rho V_2(\x)\big\}
  \end{align*} denote the sets where function $V_{i}$ increases by at
  least a factor of $\rho$ along solutions generated by $T_{A_{3-i},B}$.

  If $\rho > \cos^2\theta_2$ then
  \begin{align}
    \label{eq:5} \mathcal{B}_1(\rho) &= B\left(\p_1 +
                                       \frac{\cos\theta_2\sin\theta_2}{\rho-\cos^2\theta_2}\e_2, \frac{\sqrt
                                       \rho\sin\theta_2}{\rho-\cos^2\theta_2}\,\right) \\ \intertext{and if
    $\rho >\cos^2\theta_1$ then}
    \label{eq:6} \mathcal{B}_2(\rho) &= B\left(\p_2 -
                                       \frac{\cos\theta_1\sin\theta_1}{\rho-\cos^2\theta_1}\e_2, \frac{\sqrt
                                       \rho\sin\theta_1}{\rho-\cos^2\theta_1}\,\right),
  \end{align} that is, the sets $\mathcal{B}_{i}$ are open balls.

  If, moreover, $\rho \ge (1+\sin\theta_{1})(1+\sin\theta_{2})$, then
  \begin{align}
    \label{eq:7} \mathcal{B}_{1}(\rho) & \subseteq D_1,\\
    \intertext{and}
    \label{eq:8} \mathcal{B}_{2}(\rho) & \subseteq D_2,
  \end{align} that is, the region where $V_{i}$ increases along
  solutions generated by $T_{A_{3-i},B}$ is completely contained in
  $D_{i}$, the set where $T_{A,B}$ is singleton-valued and coincides
  with $T_{A_{i},B}$.
\end{proposition} Observe that
$(1+\sin\theta_{1})(1+\sin\theta_{2})>1\geq \cos^{2}\theta_{i}$ for
$i=1,2$ in our setting.

The proof of Proposition~\ref{prop:thm-a-helper} can be found in
Appendix~\ref{app:proof-prop}.  We can now state the main result of this
section. Beforehand we should point out that local convergence of the
Douglas--Rachford iteration is already guaranteed
by~\cite{MR3227480}. Our result implies \emph{global} convergence,
despite the complex geometry of the regions of attraction of the
individual fixed points, see Fig.~\ref{figRegAttrct}.
\begin{figure}[htb] \centering
  \includegraphics[width=.65\linewidth]{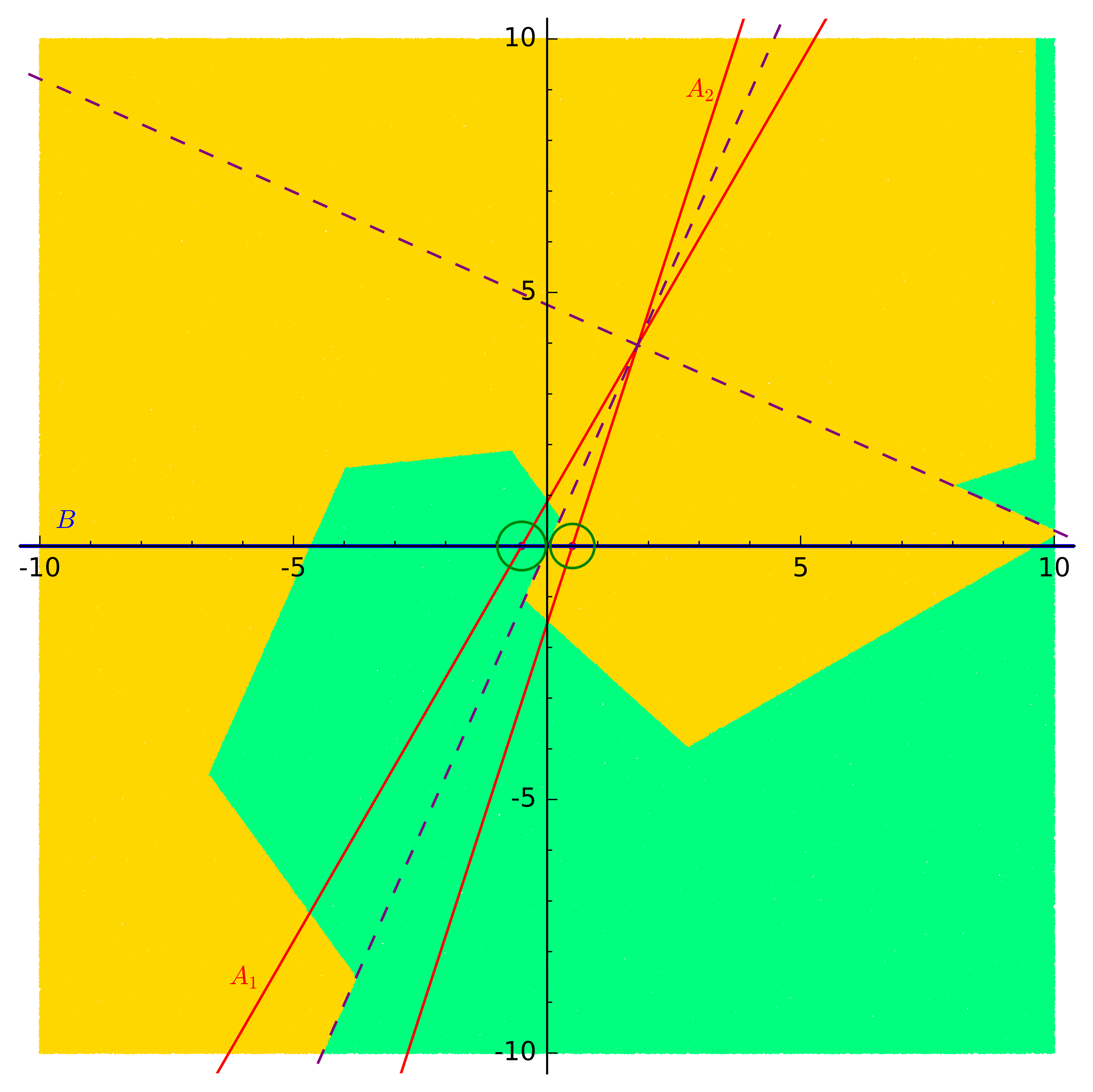}
  \caption{Regions of attraction for the case $\theta_1 = \pi/3$,
    $\theta_2 = 2\pi/5$ (for which condition~\eqref{eq:26} holds). The
    green circles are centered at $\p_1$, $\p_2$, with radii
    $d(\p_1,D_3)$, $d(\p_2,D_3)$, respectively. The figure is based on a
    simulation with about 1.5 million data points. For each randomly
    chosen initial condition a corresponding solution is computed until it
    enters the inside of one of the regions enclosed by the green circles
    (which are sub-level sets of $V_{i}$ and completely contained in
    $D_{i}$, thus invariant under $T_{A,B}$), at which point necessarily
    the solution converges to the respective intersection point
    $\p_{i}$. The initial starting point is then colored accordingly.}
  \label{figRegAttrct}
\end{figure}
\begin{theorem}\label{main thm} Suppose that either
  $\theta_{1}=\pi/2$, $\theta_{2}=\pi/2$, or that
  \begin{equation}
    \label{eq:26} \Big(\log
    \big((1+\sin\theta_{1})(1+\sin\theta_{2})\big)\Big)^{2} < \log
    (\cos^{2}\theta_{1}) \log (\cos^{2}\theta_{2}).
  \end{equation} Then there exist $\alpha \in ]0,\infty[$ and $\gamma
  \in ]0,1[$ such that the function $V\colon \R^2\to \R_+$, defined by
  \begin{align}\label{def V} V(\x) \coloneqq V_1(\x)^\alpha V_2(\x),
  \end{align} satisfies
  \begin{itemize}
  \item inequalities~\eqref{cond omegas} with
    \begin{align}
      \label{eq:9} \omega_1(\x) & \coloneqq \min\{\|\x-\p_1\|,
                                  \|\x-\p_2\|\} = d(\x,A\cap B),\\
      \label{eq:10} \omega_2(\x) & \coloneqq \max\{\|\x-\p_1\|,
                                   \|\x-\p_2\|\},\\
      \label{eq:11} \varphi_{1}(r) & \coloneqq \varphi_{2}(r)
                                     \coloneqq r^{2\alpha+2};
    \end{align}
  \item the decrease condition
    \begin{align}
      \label{multi decrease V} \sup_{\y \in T_{A,B}\x}V(\y) \le \gamma
      V(\x);
    \end{align}
  \item as well as $V(\x)=0$ if and only if $\x\in\mathcal{A}\coloneqq
    \{\p_{1},\p_{2}\}$.
  \end{itemize} That is, the Douglas--Rachford iteration~\eqref{eq:1}
  is robustly $\kl$-stable with respect to $(\omega_{1},\omega_{2})$.
\end{theorem}

\begin{proof} First we establish that there exist $\alpha \in
  ]0,\infty[$ and $\gamma \in ]0,1[$ such that
  \begin{align}\label{the condition}
    \begin{aligned}
      &&(\cos^{2}\theta_1)^\alpha\big((1+\sin\theta_{1})(1+\sin\theta_{2})\big)\phantom{^\alpha}
      &\le \gamma \\ &\text {and}\\
      &&(\cos^{2}\theta_2)\phantom{^\alpha}\big((1+\sin\theta_{1})(1+\sin\theta_{2})\big)^\alpha
      &\le \gamma.
    \end{aligned}
  \end{align} Fig.~\ref{FigureDiscsGoodBad} visualizes the
  relationship between condition~\eqref{the condition} and
  assertion~\eqref{multi decrease V} in Theorem~\ref{main thm}.
  \begin{figure}[htb!]  \centering \centering
    \includegraphics[width=.495\linewidth]{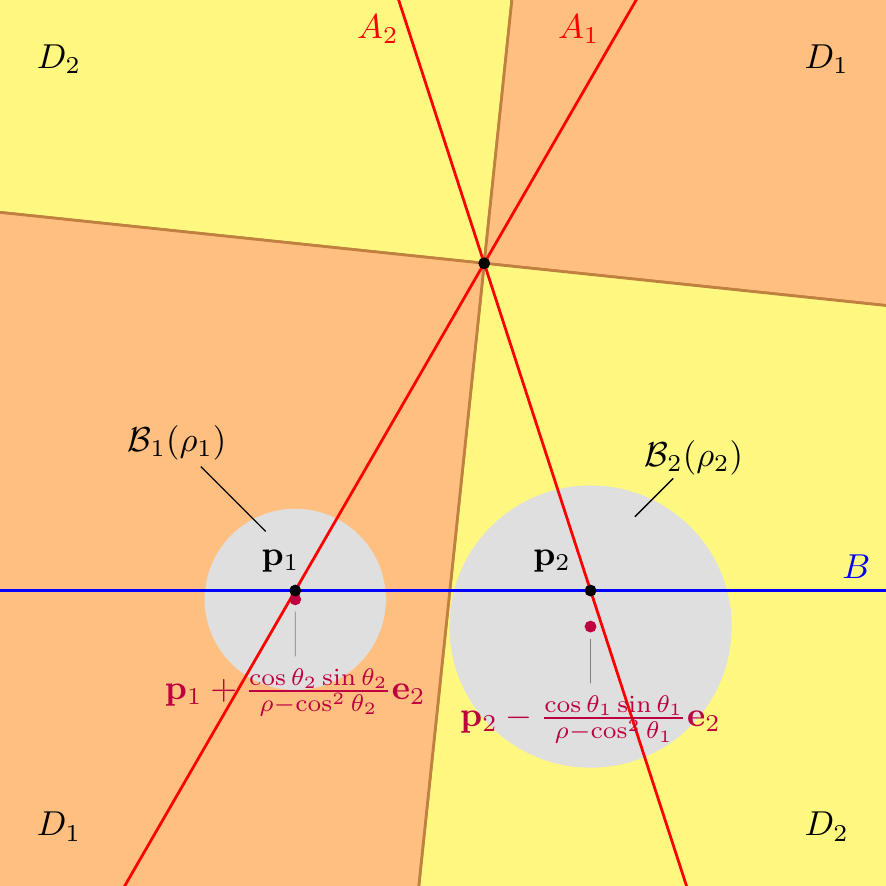} \hfill
    \includegraphics[width=.495\linewidth]{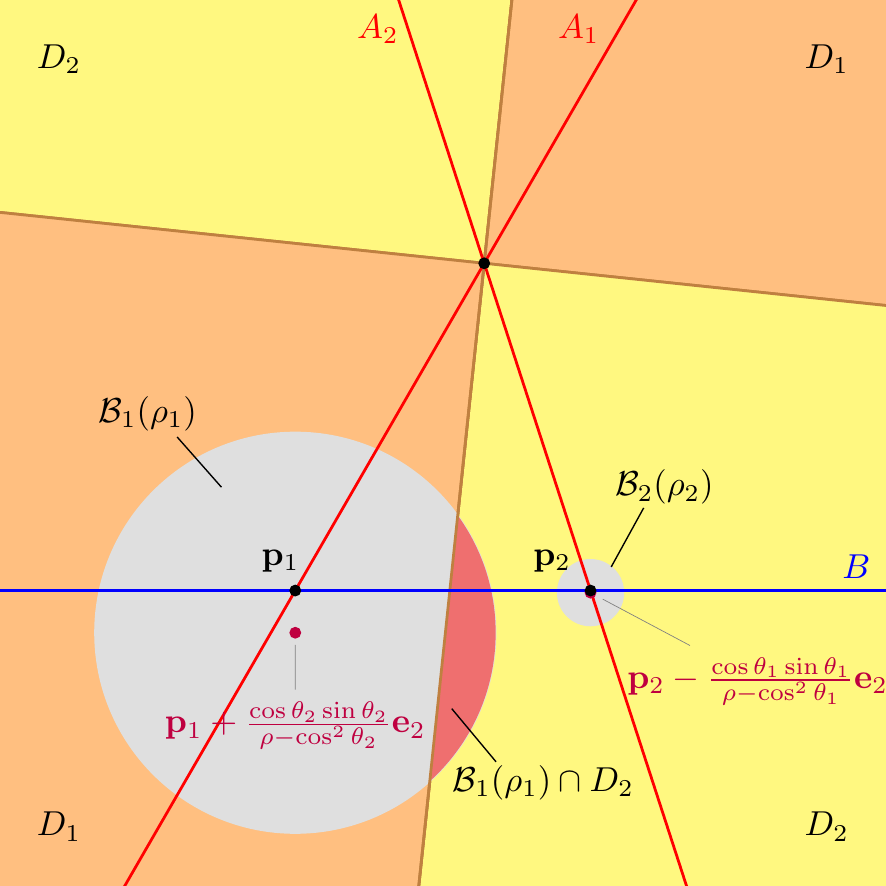}
    \caption{The open balls $\mathcal{B}_1(\rho_1)$,
      $\mathcal{B}_2(\rho_2)$, defined in
      Proposition~\ref{prop:thm-a-helper}, for two different choices of
      $\alpha$ ($\alpha=1$ on the left and $\alpha=3$ on the right),
      $\theta_{1}=\pi/3$, $\theta_{2}=3\pi/5$ and $\gamma=0.95$.  Here
      $\rho_1 = \left(\frac{\gamma}{\cos^2\theta_2}\right)^{1/\alpha}$ and
      $\rho_2 = \gamma\left(\frac{1}{\cos^2\theta_1}\right)^{\alpha}$.  In
      the left figure conditions~\eqref{the condition} both hold, while in
      the right figure the second condition is violated. In the left figure
      the function $V$ in~\eqref{def V} satisfies~\eqref{multi decrease V}
      everywhere, while in the right figure it does not. The red slice of
      the open ball in the right figure is the set of points where
      $V(\y)>\gamma V(\x)$ for $\y\in T_{A,B}\,\x$.
      \label{FigureDiscsGoodBad} }
  \end{figure} Inequalities~\eqref{the condition} trivially hold if
  $\theta_{1}=\pi/2$ or $\theta_{2}=\pi/2$, so if $\theta_{1}\ne\pi/2$
  and $\theta_{2}\ne\pi/2$ then for condition~\eqref{the condition} to
  hold it is necessary and sufficient that
  $(\cos^{2}\theta_{1})^{\alpha}\big((1+\sin\theta_{1})(1+\sin\theta_{2})\big)<1$
  and simultaneously
  $(\cos^{2}\theta_{2})\big((1+\sin\theta_{1})(1+\sin\theta_{2})\big)^{\alpha}<1$,
  which in turn is equivalent to \eqref{eq:26}.
  
  The first claim about the functions defined in~\eqref{eq:9},
  \eqref{eq:10}, and \eqref{eq:11} satisfying~\eqref{cond omegas}
  follows by direct computation and the definition~\eqref{def V} of
  $V$. Obviously the functions $\varphi_{i}$ are of class $\K_{\infty}$
  as $\alpha>0$.
  
  From its definition, $V(\x)=0$ holds if and only if $\x$ is either
  $\p_{1}$ or $\p_{2}$. This establishes the third claim.

  To establish the second claim, i.e., the decrease condition
  \eqref{multi decrease V}, we have to consider several cases.  The
  first is that $\x \in \{\p_1,\p_2\}$.  Then $V(\x)$ $=$ $V(T_{A,B}\x)
  = 0$ and so~\eqref{multi decrease V} holds.

  Consider next the case that $\x\in D_1\setminus\{\p_1,\p_2\}$. In
  this case $T_{A,B}=T_{A_{1},B}$ is single-valued by~\eqref{def T}.  We
  find
  \begin{eqnarray}\label{equal with cos} \nonumber V(T_{A,B}\x) & = &
                                                                      V(T_{A_1,B}\x) = V_1^\alpha(T_{A_1,B}\x) V_2(T_{A_1,B}\x) \\ & = &
                                                                                                                                         \cos^{2\alpha}\theta_1V^\alpha_1(\x)V_2(T_{A_1,B}\x),
  \end{eqnarray} where in the second line we have used \eqref{eq:2} of
  Proposition~\ref{prop decay V's}. Now, if $\theta_{1}=\nfrac\pi2$ then
  this reads $V(T_{A,B}\x)=0\leq \gamma V(\x)$, i.e., the proof for the
  case $\x\in D_{1}\setminus\{\p_1,\p_2\}$ is complete. So in the
  following assume that $\theta_{1}\ne \nfrac\pi2$ and note that we also
  have $V_1(\x)>0$ since we assumed that $\x\ne \p_{1}$. By way of
  contradiction, assume now that we have
  \begin{align}\label{cond bigger} V(T_{A,B}\x) > \gamma V(\x).
  \end{align} Then we can arrange~\eqref{equal with cos} into
  \begin{eqnarray*} V_2(T_{A_1,B}\x) & = &
                                           \left(\frac{1}{\cos^2\theta_1}\right)^\alpha\cdot
                                           \frac{V(T_{A,B}\x)}{V_1^\alpha(\x)} \stackrel{\eqref{cond bigger}}{>}
                                           \left(\frac{1}{\cos^2\theta_1}\right)^\alpha\gamma\cdot
                                           \frac{V(\x)}{V_1^\alpha(\x)} \\ & \stackrel{\eqref{def V}}{=} &
                                                                                                           \left(\frac{1}{\cos^2\theta_1}\right)^\alpha \gamma V_2(\x)
                                                                                                           \stackrel{\eqref{the condition}}{\ge}
                                                                                                           (1+\sin\theta_{1})(1+\sin\theta_{2}) V_2(\x).
  \end{eqnarray*} An application of
  Proposition~\ref{prop:thm-a-helper} lets us deduce that $\x$ must be
  in $D_{2}$. However, the sets $D_{1}$ and $D_{2}$ are disjoint, so
  this contradicts our assumptions.  This means that the condition
  $V(T_{A,B}\x) > \gamma V(\x)$ cannot hold and we have $V(T_{A,B}\x)
  \le \gamma V(\x)$.

  The case $\x \in D_2\setminus\{\p_1,\p_2\}$ is analogous to the
  previous one and is thus omitted.
  
  Finally, assume that $\x \in D_3\setminus\{\p_1,\p_2\}$. Then
  by~\eqref{def T} we have $T_{A,B}\x = \big\{T_{A_1,B}\x,
  T_{A_2,B}\x\big\}$.  Now, if, by way of contradiction, we assume
  $V(T_{A_1,B}\x) > \gamma V(\x)$ then as before it must follow that
  $\x\in D_2$. If $V(T_{A_2,B}\x) > \gamma V(\x)$ then it must follow
  that $\x \in D_1$. In both cases we get $\x \notin D_3$, and so
  neither of these inequalities can hold true. We therefore have in this
  case $V(T_{A_1,B}\x) \le \gamma V(\x)$ and $V(T_{A_2,B}\x) \le \gamma
  V(\x)$. We have established that inequality~\eqref{multi decrease V},
  respectively,~\eqref{cond decay}, holds for all $\x\in \R^2$.

  This establishes that $V$ is indeed a global Lyapunov function
  for~\eqref{eq:1} with respect to $(\omega_{1},\omega_{2})$, and by
  Theorem~\ref{thm KT} it follows that the difference
  inclusion~\eqref{eq:1} is robustly $\kl$-stable.
\end{proof}

\begin{remark}
  \label{rem:simpler-condition} Numerical evidence suggests that the
  region in the parameter space $\{(\theta_{1},\theta_{2})\colon$
  $0<\theta_{1}\leq\pi/2$, $\theta_{1}<\theta_{2}<\pi \}$ where all
  solutions of the Douglas--Rachford iteration converge to
  $\{\p_{1},\p_{2}\}$ is bigger than the set shown in
  Fig.~\ref{subfig:a}, while for some parameter combinations that are
  outside this set, the Douglas--Rachford iteration may get caught by
  attractive periodic orbits, cf.\ Figs.~\ref{subfig:b}--\ref{subfig:d}.
\end{remark}

\begin{figure}[htbp] \centering \def\figwidth{.475\linewidth}
  \subfloat[The different regions (red) of parameters for which not all
  solutions converge to a fixed point of $T_{A,B}$, relative to the
  admissible parameters (green) and region where~\eqref{eq:26} holds
  (blue).]{\label{subfig:a}\includegraphics[width=\figwidth]{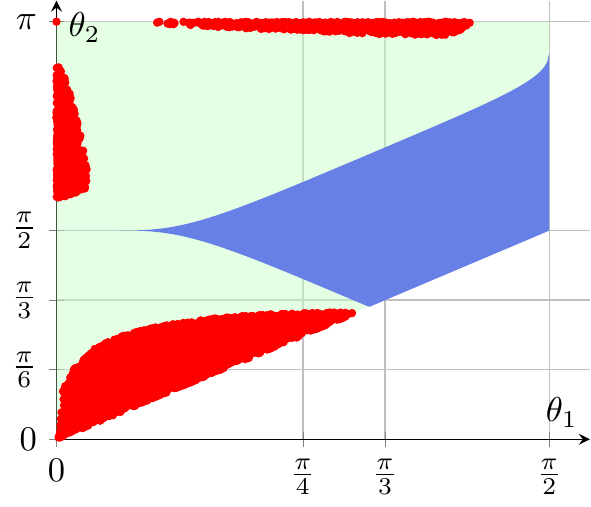}}\hfill
  \subfloat[ The parameters $\theta_{1}=0.703469$, $\theta_{2}=3.138852$
  admit a periodic orbit with period length 1410 containing
  $\x_{0}=(0.392560, -0.351588)$.
  ]{\label{subfig:b}\includegraphics[width=\figwidth]{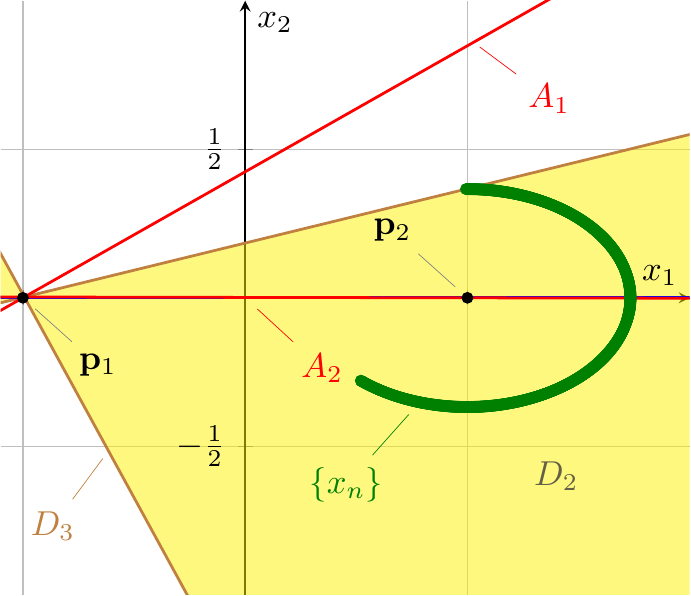}}\\%
  \subfloat[ The parameters $\theta_{1}=0.082719$, $\theta_{2}=2.064601$
  admit a periodic orbit with period length 58 containing
  $\x_{0}=(-0.123641, -0.510395)$.
  ]{\label{subfig:c}\includegraphics[width=\figwidth]{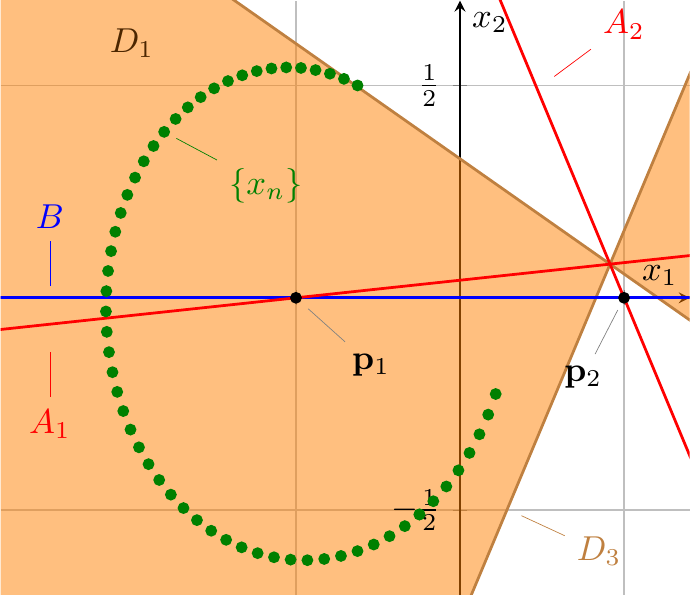}}\hfill
  \subfloat[ The parameters $\theta_{1}=0.748491$, $\theta_{2}=0.772301$
  admit a periodic orbit with period length 2 containing
  $\x_{0}=(0.101912, 0.189275)$.
  ]{\label{subfig:d}\includegraphics[width=\figwidth]{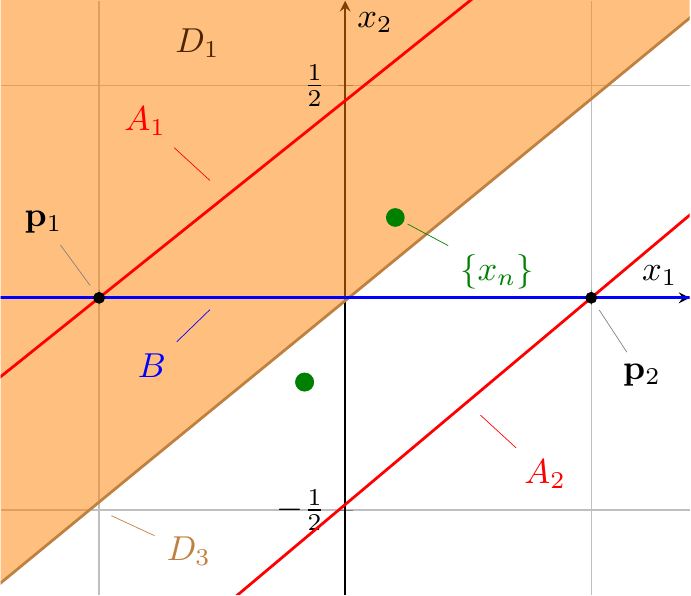}}
  \caption{Numerical experiments. In Fig.~\ref{subfig:a} we see
    regions in the $(\theta_{1},\theta_{2})$ plane (restricted to
    admissible pairs) of parameter combinations for which not all
    solutions converge to $\{\p_{1},\p_{2}\}$.  A sample solution (green)
    for the lump of points in the top right of the plot is shown in
    Fig.~\ref{subfig:b}, as typical solution from the region on the left
    in Fig.~\ref{subfig:c}, and one from the region closest to the
    $\theta_{1}$-axis in Fig.~\ref{subfig:d}.  }
  \label{fig:periodic-orbits}
\end{figure}

Next, we discuss how the order of reflections $R_{A}$ and $R_{B}$
affects the Lyapunov function construction in this paper.

\begin{corollary}
  \label{cor:reverse-order-T-BA} Under the same assumptions as in
  Theorem~\ref{main thm}, the same function $V$ given in~\eqref{def V}
  satisfies the same conclusions for the\linebreak[4] Douglas--Rachford
  iteration given by
  \begin{equation*} \z^{+}\in T_{B,A}\,\z.
  \end{equation*} In other words, for this particular geometry the
  order of the reflections in the Douglas--Rachford iteration does not
  affect its robust stability.
\end{corollary}

\begin{proof} By \cite[Proposition~2.5~(i) and
  Lemma~2.4~(iii)]{bauschkemoursi2016-on-the-order-of-the-operators-in-the-douglas-rachford-algorithm}
  we have
  \begin{equation} T_{A,B}=R_{B}T_{B,A}R_{B}.\label{eq:18}
  \end{equation} Noting that $R_{B}^{-1}=R_{B}$ and that
  \begin{equation}
    \label{eq:17} V_{i}(R_{B}\x)=V_{i}(\x)
  \end{equation} for all $\x\in\R^{2}$ and $i=1,2$, we only need to
  verify the decrease condition.

  From~\eqref{multi decrease V} we have that
  $$
  \sup_{\y \in T_{A,B}\x}V(\y) \le \gamma V(\x).
  $$
  Let $\w \in T_{B,A}\z$ with $\z=R_{B}\x$. We want to show that
  $V(\w)\leq \gamma V(\z)$.

  To this end note that with $\y=R_{B}\w$, we have
  \begin{align*} V(\w) &\stackrel{\eqref{eq:17}}{=} V(R_{B}\w) =
                         V(\y),\\ \intertext{where clearly $\y\in
    R_{B}T_{B,A}R_{B}\x\stackrel{\eqref{eq:18}}{=}T_{A,B}\x$. Hence we
    continue to estimate} & \stackrel{\eqref{multi decrease V}}{\leq}
                            \gamma V(\x) = \gamma V(R_{B}\z) \stackrel{\eqref{eq:17}}{=} \gamma
                            V(\z).
  \end{align*} This establishes the decrease condition. All other
  estimates are the same as in the theorem.
\end{proof}

Theorem~\ref{main thm} allows us to specify explicitly the convergence
behavior of the Douglas--Rachford difference inclusion~\eqref{eq:1}
even in the presence of perturbations. For example, it is possible to
prove the following robustness result.

\begin{corollary}
  \label{cor:robust-kl-convergence-rate} Under the assumptions of
  Theorem~\ref{main thm}, and with $\eps \in ]0,1[$ such that
  $(1+\eps)^2\gamma < 1$, let
  \begin{align}\label{def sigma}
    \sigma(\x)  = \big((1+\eps)^{\frac{1}{2(1+\alpha)}}-1\big) d(\x,A\cap B).
  \end{align} Then for all $\x\in \R^2$ and $n \in \Z_+$,
  \begin{align*} \sup_{\phi \in
    \S_\sigma(\x,T_{A,B})}d(\phi(\x,n),A\cap B) \le \max\{\|\x-\p_1\|,
    \|\x-\p_2\|\}\left((1+\eps)^2\gamma\right)^{\frac{n}{2\alpha+2}}.
  \end{align*}
\end{corollary} 

The proof of Corollary~\ref{cor:robust-kl-convergence-rate} can be found in Appendix~\ref{app:proof-cor}.


\section{Perspectives and Open Problems}
\label{sec:persp-open-probl}

While the Lyapunov approach for studying asymptotic stability has a
long history, its use to study convergence of Douglas--Rachford is
very recent. Two main reasons for its success are that, in essence,
asymptotic stability implies the existence of a Lyapunov function,
and, secondly, that these functions can provide global stability
certificates and, in the non-global case, useful estimates on the
regions of attraction, i.e., the sets of initial conditions from where
the iteration is going to converge. However, several problems are left
open for further investigation.

One question concerning Theorem~\ref{main thm} is to find a Lyapunov
function for a larger region in the $(\theta_1, \theta_2)$-domain,
cf.~Fig.~\ref{subfig:a}, in order to reduce the conservativeness
inherent to the present approach. Notice that in the case where
$\theta_1 \in ]0,\pi/2]$ and $\theta_2 = \pi-\theta_1$, the function
$V\colon \R^2\to \R_+$, defined by
\begin{align}\label{def V min} V(\x) \coloneqq
  \min\{V_1(\x),V_2(\x)\},
\end{align} satisfies $V(T_{A,B}\x) = \cos^2\theta_1V(\x)$. This case
is particularly simple, since we have $T_{A_i,B}D_i \subseteq D_i$ for
$i=1,2$. However, if $V$ is chosen as in~\eqref{def V min} then it
does not satisfy the decay condition~\eqref{multi decrease V} for some
of the choices of $\theta_1$ and $\theta_2$ that satisfy the
assumptions of Theorem~\ref{main thm}.

A natural extension of Theorem~\ref{main thm} concerns the study of
affine subspaces in higher dimensional spaces, along the lines of
\cite{MR3233066}, which could provide a better intuition for
understanding the global convergence of the Douglas--Rachford
iteration in more general scenarios.

Yet another question concerns the study of the parameter regions where
periodic orbits seem to occur. In numerical experiments these periodic
orbits appear to be attracting nearby solutions, so it seems
reasonable to conjecture that for the periodic orbits, too, one can
find suitable (local) Lyapunov functions and then use these to
estimate the corresponding regions of attraction, which are linked to
the success rate of the algorithm (ratio of convergent/nonconvergent
solutions).

As for robustness, we did not try to give optimal bounds in
Corollary~\ref{cor:robust-kl-convergence-rate}, and there may be room
for further improvement.

For us the most interesting question is to construct a Lyapunov
function for the case of a polygon and a line, which opens a pathway
towards considering even more complex geometries like circle and line
or ellipse and lines as limits of polygons, possibly exploiting
robustness properties along the way.  If we consider the case $\H =
\R^2$, then at any given point the Douglas--Rachford operator reflects
either with respect to two lines or with respect to a line and a
point, see Fig.~\ref{figPoly}.  This question requires a better
understanding of the Douglas--Rachford operator in the case of
multiple lines than we currently have.

\begin{figure}[htb!]  \centering
  \includegraphics[width=.65\linewidth]{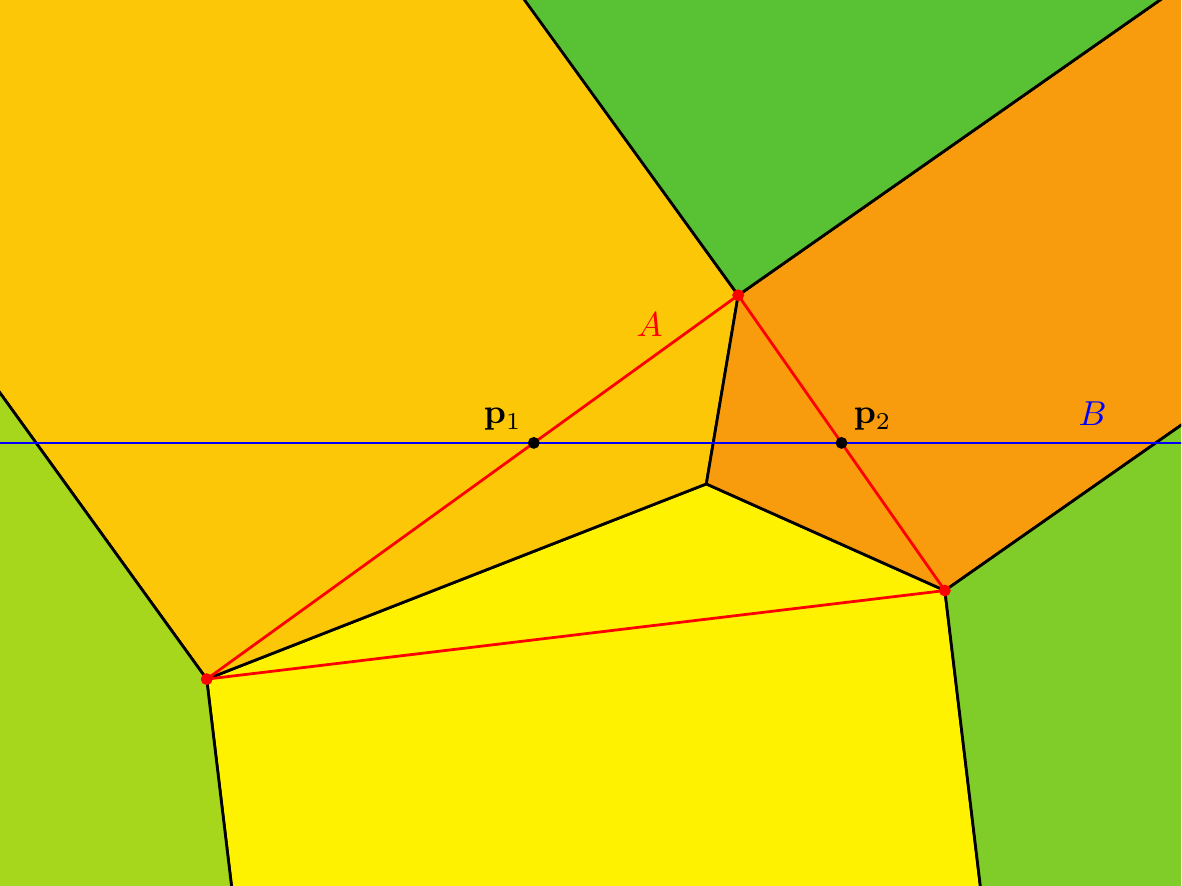}
  \caption{The Douglas--Rachford iteration for a triangle (as a simple
    polygon) and a straight line. At each point we reflect either with
    respect to the blue line and one of the red lines (the yellow-orange
    domains) or with respect to the blue line and one of the red points
    (the green domains). The black lines are where the map is
    multi-valued.}
  \label{figPoly}
\end{figure}

The more general case of intersecting non-convex sets that are
themselves finite unions of convex sets is understood
locally~\cite{MR3227480}. However, a Lyapunov approach could shed
light on the region of attraction and lead to the important insight
what other (other than convex) conditions ensure global convergence or
at least a large region of attraction, which is of interest in
practice.

Lastly, a seemingly simple scenario, that was kindly brought to our
attention by Heinz Bauschke, is the convergence behavior in the case
of two finite sets. In this case projections are very easy to compute,
but the resulting dynamics can be very rich, and essentially nothing
is known about the convergence behavior of the Douglas--Rachford
iteration to date. This problem, too, once understood, could at a
larger scale (many points!) be used to approximate more complex
non-convex cases and provide vital insights to their understanding. To
put this into the words of the late Jon Borwein: ``If there is a
problem you don't understand, there's a smaller problem within it that
you don't understand. So solve that one first.'' The idea here is, of
course, that the simpler problem is easier and its solution provides
crucial insight into the bigger problem.

\section{Conclusions}
\label{sec:conclusions}

This paper presents an explicit construction of a Lyapunov function
for a Douglas--Rachford iteration in a non-convex setting by combining
simple, local Lyapunov functions to a global Lyapunov function. It is
discussed how the existence of a global Lyapunov function demonstrates
not only global converge to one of the intersection points, but also
implies strong stability and robustness properties of the
Douglas--Rachford iteration. Several leads for further research
directions are provided.

\appendix

\section{Proof of Proposition~\ref{prop:thm-a-helper}.}
\label{app:proof-prop}

We begin by establishing condition~\eqref{eq:5}. The proof for
condition~\eqref{eq:6} is essentially the same and thus omitted for
brevity.

\subsubsection*{Establishing~Condition~\eqref{eq:5}} We have by
Proposition~\ref{prop decay V's} that
$$
T_{A_2,B}\x = \p_2 + \cos\theta_{2}\M_{\theta_{2}}\left(\x-\p_2\right)
= \nfrac{1}{2} \e_1 +
\cos\theta_{2}\M_{\theta_{2}}\left(\x-\nfrac{1}{2} \e_1\right).
$$
If $\theta_{2}=\nfrac\pi2$ then this simplifies further to
$T_{A_2,B}\x = \p_2$, resulting in $V_{1}(T_{A_2,B}\x) = 1$. We can
hence deduce that
\begin{align} 1= V_1(T_{A_2,B}\x) > \rho V_1(\x) = \rho
  \|\x-\p_{1}\|^{2}\nonumber\\ \iff \quad \x\in
  B(\p_{1},\frac{\sqrt\rho}{\rho}).\label{eq:12}
\end{align} Now, if $\theta_{2}\ne\nfrac\pi2$ then we have
\begin{multline*} V_1(T_{A_2,B}\x) =$ $\left\|T_{A_2,B}\x -
    \p_1\right\|^2 = \left\| T_{A_2,B}\x + \nfrac{1}{2} \e_1\right\|^2 \\
  = \left\| \e_1 + \cos\theta_2\M_{\theta_{2}}\left(\x-\nfrac{1}{2}
      \e_1\right)\right\|^2 =
  \left\|\cos\theta_{2}\M_{\theta_{2}}\left(\x-\nfrac{1}{2} \e_1 +
      S_{\theta_2}\e_1\right)\right\|^2 \\ = \cos^2\theta_2
  \left\|\x-\nfrac{1}{2} \e_1 +S_{\theta_2}\e_1\right\|^2,
\end{multline*} where $S_{\theta_2} \coloneqq \frac{1}{\cos\theta_2}
\M_{-\theta_{2}}$ satisfies $S_{\theta_2}\e_1 =
\e_1+\tan\theta_2\e_2$, and so $-\nfrac{1}{2} \e_1 + S_{\theta_2}\e_1
= \nfrac{1}{2} \e_1 + \tan\theta_2\e_2$.  The inequality
$V_1(T_{A_2,B}\x) > \rho V_1(\x)$ is thus equivalent to
\begin{align*}
  \cos^2\theta_2\left\|\x+\nfrac{1}{2} \e_1 +\tan\theta_2\e_2\right\|^2
> \rho\left\|\x+\nfrac{1}{2} \e_1\right\|^2.
\end{align*}  
Expanding this gives
\begin{align*}
  \cos^2\theta_2\|\x+\nfrac{1}{2} \e_1\|^2 + 
2\cos^2\theta_2\tan\theta_2\langle \x + \nfrac{1}{2} \e_1,
\e_2\rangle + \cos^2\theta_2\tan^2\theta_2 >
\rho\left\|\x+\nfrac{1}{2} \e_1\right\|^2
\end{align*}
or, equivalently,
\begin{align}\label{bound before complete}
  (\rho-\cos^2\theta_2)\left\|\x+\nfrac{1}{2} \e_1\right\|^2 -
  2\cos^2\theta_2\tan\theta_2\left\langle \x+\nfrac{1}{2} \e_1,
  \e_2\right\rangle < \sin^2\theta_2.
\end{align} By assumption we have $\rho > \cos^2\theta_2$. Hence
estimate~\eqref{bound before complete} is equivalent to
\begin{align*} \left\|\x+\nfrac{1}{2} \e_1\right\|^2 -
  2\ofrac{\cos^2\theta_2\tan\theta_2}{\rho-\cos^2\theta_2}\left\langle
  \x+\nfrac{1}{2} \e_1, \e_2\right\rangle <
  \ofrac{\sin^2\theta_2}{\rho-\cos^2\theta_2}.
\end{align*} Completing the square gives
\begin{align}\label{norm is small} \nonumber \left\|\x+\nfrac{1}{2}
  \e_1 -
  \ofrac{\cos\theta_2\sin\theta_2}{\rho-\cos^2\theta_2}\e_2\right\|^2 &
                                                                        <
                                                                        \ofrac{\sin^2\theta_2}{\rho-\cos^2\theta_2}+\ofrac{\cos^2\theta_2\sin^2\theta_2}{(\rho-\cos^2\theta_2)^2}
  \\ & = \ofrac{\rho \sin^2\theta_2}{(\rho-\cos^2\theta_2)^2}.
\end{align} Since $\theta_2 \in ]0,\pi[$, we have $\sin\theta_2 > 0$
and so \eqref{norm is small} is equivalent to
\begin{align*} \x \in B\left(-\nfrac{1}{2} \e_1 +
  \ofrac{\cos\theta_2\sin\theta_2}{\rho-\cos^2\theta_2}\e_2,
  \ofrac{\sqrt \rho \sin\theta_2}{\rho-\cos^2\theta_2}\,\right).
\end{align*} Since $\p_1 = -\nfrac{1}{2} \e_1$, this
establishes~\eqref{eq:5}, which contains \eqref{eq:12} as a special
case.\qede

\subsubsection*{An Auxiliary Lemma} \let\frac=\ofrac

Before we can proceed with the proof of the proposition, we need the
following auxiliary result, which provides a characterization of the
set $D_3$ defined in~\eqref{def D's}. Note that since $A_1$ and $A_2$
are two non-parallel straight lines, their intersection is a single
point.

\begin{lemma}
  \label{aux-lemma-D3} Let $\c$ be the unique intersection point of
  $A_1$ and $A_2$. Then
  \begin{align}
    \label{for c}
    \c=\left(\frac{\sin(\theta_1+\theta_2)}{2\sin(\theta_2-\theta_1)},\frac{\sin\theta_1\sin\theta_2}{\sin(\theta_2-\theta_1)}\right)
  \end{align} and we have
  \begin{align*} D_3 = \left\{ \x\in \R^2\colon \langle
    \x-\c,\n_1\rangle = 0\right\} \cup \left\{ \x\in \R^2\colon \langle
    \x-\c,\n_2\rangle = 0\right\},
  \end{align*} where $\n_1$, $\n_2$, are given by
  \begin{align}
    \label{for n1} \n_1 & =
                          \left(\cos\left(\frac{\theta_1+\theta_2}{2}\right),\sin\left(\frac{\theta_1+\theta_2}{2}\right)\right),
    \\
    \label{for n2}\n_2 & =
                         \left(\sin\left(\frac{\theta_1+\theta_2}{2}\right),-\cos\left(\frac{\theta_1+\theta_2}{2}\right)\right).
  \end{align}
\end{lemma}

\begin{proof} The line $A_1$ is the collection of all points $\x\in
  \R^2$ that satisfy $\langle \x,\e_1\rangle \sin\theta_1$ $-$
  $\langle\x,\e_2\rangle \cos\theta_1$ $+$ $\nfrac{1}{2} \sin\theta_1 =
  0$.  Similarly, the line $A_2$ is the collection of all points $\x\in
  \R^2$ that satisfy $\langle \x,\e_1\rangle \sin\theta_2 -
  \langle\x,\e_2\rangle \cos\theta_2 - \nfrac{1}{2}\sin\theta_2 = 0$.
  Solving these two equations implies that the intersection point $\c$
  between $A_1$ and $A_2$ is indeed given by~\eqref{for c}.  Now, the
  (normalized) normal vectors to the lines splitting the angles between
  $A_1$ and $A_2$ are given by~\eqref{for n1} and \eqref{for n2} and
  this completes the proof of the auxiliary lemma.
\end{proof} \medskip
\noindent We are now in a position to establish
condition~\eqref{eq:7}.  The proof of condition~\eqref{eq:8} follows
closely that of condition~\eqref{eq:7} and is thus omitted for reasons
of space.

\subsubsection*{Establishing Condition~\eqref{eq:7}}
\def\d{\mathbf{d}}

Since the direction vector of $A_i$ is\linebreak[4] $(\cos\theta_i,
\sin\theta_i)$, $\d_{i}^{\perp}\coloneqq(\sin\theta_i, -\cos\theta_i)$
is a normal vector to $A_{i}$ and the distance of a point $Q$ to
$A_{i}$ is given by $|\langle Q-\p_{i},\d_{i}^{\perp}\rangle|$. We
compute
\begin{align}
  d\left(\p_1+\frac{\cos\theta_2\sin\theta_2}{\rho-\cos^2\theta_2}\e_2,
  A_1\right) & =
               \left|\frac{\cos\theta_1\cos\theta_2\sin\theta_2}{\rho-\cos^2\theta_2}\right|
               \label{eq:13}
\end{align} and
\begin{align}
  d\left(\p_1+\frac{\cos\theta_2\sin\theta_2}{\rho-\cos^2\theta_2}\e_2,
  A_2\right) & =
               \left|\frac{\cos^2\theta_2\sin\theta_2}{\rho-\cos^2\theta_2}+\sin\theta_2\right|.
               \label{eq:14}
\end{align} Now, since $\theta_1 \in ]0,\pi/2]$ and $\theta_2\in
]\theta_{1},\pi[$, we have $|\cos\theta_1\cos\theta_2| < 1$,
$\cos^{2}\theta_{i}<1$, $\sin\theta_2 > 0$, and
$(1+\sin\theta_{1})(1+\sin\theta_{2}) > 1$. Since we assumed that
$\rho \ge (1+\sin\theta_{1})(1+\sin\theta_{2})$, we have
$|\cos\theta_1\cos\theta_2| < 1 < \rho$. With these estimates we can
bound~\eqref{eq:13} generously as
\begin{align}\label{p1 dist to A1}
  d\left(\p_1+\frac{\cos\theta_2\sin\theta_2}{\rho-\cos^2\theta_2}\e_2,
  A_1\right) < \frac{\rho\sin\theta_2}{\rho-\cos^2\theta_2}
\end{align} and simplify~\eqref{eq:14} to
\begin{align}\label{p1 dist to A2}
  d\left(\p_1+\frac{\cos\theta_2\sin\theta_2}{\rho-\cos^2\theta_2}\e_2,
  A_2\right) =
  \sin\theta_2\left(\frac{\cos^2\theta_2}{\rho-\cos^2\theta_2} +
  1\right) = \frac{\rho\sin\theta_2}{\rho-\cos^2\theta_2}.
\end{align} In light of~\eqref{p1 dist to A1} and~\eqref{p1 dist to
  A2}, $A_{1}$ is the closer line to
$\p_1+\frac{\cos\theta_2\sin\theta_2}{\rho-\cos^2\theta_2}\e_2$, so it
follows that
$\p_1+\frac{\cos\theta_2\sin\theta_2}{\rho-\cos^2\theta_2}\e_2 \in
D_1$. Therefore, in order to prove~\eqref{eq:7}, it is enough to show
that
\begin{align}\label{want dist small 1} \frac{\sqrt \rho
  \sin\theta_2}{\rho-\cos^2\theta_2} \leq
  d\left(\p_1+\frac{\cos\theta_2\sin\theta_2}{\rho-\cos^2\theta_2}\e_2,
  D_3\right),
\end{align} that is, we want the radius of the ball to be smaller than
the distance of the center to the boundary of $D_1$ (which is exactly
$D_3$). Now, by the auxiliary Lemma~\ref{aux-lemma-D3} , we have
\begin{equation}\label{p1 dist to D3}
  d\left(\p_1+\frac{\cos\theta_2\sin\theta_2}{\rho-\cos^2\theta_2}\e_2,
    D_3\right) = \min_{i=1,2}\left\{\left|\left\langle
        \c-\p_1-\frac{\cos\theta_2\sin\theta_2}{\rho-\cos^2\theta_2}\e_2,
        \n_i\right\rangle\right|\right\}.
\end{equation} By squaring both sides of~\eqref{want dist small 1} and
using~\eqref{p1 dist to D3} we need to establish that
\begin{alignat}{2} \nonumber \frac{\rho
    \sin^2\theta_2}{(\rho-\cos^2\theta_2)^2} & \leq
  \min_{i=1,2}\Bigg\{&&\left(\left\langle
      \c-\p_1-\frac{\cos\theta_2\sin\theta_2}{\rho-\cos^2\theta_2}\e_2,
      \n_i\right\rangle\right)^{2}\Bigg\}.\\ \intertext{Using~\eqref{for c},
    \eqref{for n1}, \eqref{for n2}, as well as standard trigonometric
    identities, the right hand side simplifies to} & = \min\Bigg\{
  &&\Bigg(\frac{\sin\theta_2}{2\sin\left(\frac{\theta_2-\theta_1}{2}\right)}
  -
  \frac{\cos\theta_2\sin\theta_2}{\rho-\cos^2\theta_2}\sin\left(\frac{\theta_1+\theta_2}{2}\right)\Bigg)^{2}, \label{eq:20}\\
  &&&
  \Bigg(\frac{\sin\theta_2}{2\cos\left(\frac{\theta_2-\theta_1}{2}\right)}+\frac{\cos\theta_2\sin\theta_2}{\rho-\cos^2\theta_2}\cos\left(\frac{\theta_1+\theta_2}{2}\right)\Bigg)^{2}\,\Bigg\},\label{eq:21}
\end{alignat} so that we need to verify two inequalities, both of
which can be simplified further. Starting with~\eqref{eq:20}, we take
a common denominator and extract common factors. Using that
$0<\theta_{2}<\pi$, $\sin^{2}\theta_{2}>0$, and
$\rho>\cos^{2}\theta_{2}$, as well as trusty trigonometric identities,
we simplify~\eqref{eq:20} to
\begin{equation}
  \label{eq:22} \rho^2 - (2 - 2\sin\theta_1\sin\theta_2)\rho +
  \cos^2\theta_1\cos^2\theta_2 \geq 0.
\end{equation} A similar argument can be made for~\eqref{eq:21}, which
simplifies to
\begin{equation}
  \label{eq:24} \rho^2 - (2 + 2\sin\theta_1\sin\theta_2)\rho +
  \cos^2\theta_1\cos^2\theta_2 \geq 0.
\end{equation} For $\rho> 0$, the left hand side of \eqref{eq:22} is
greater than the left hand side of~\eqref{eq:24}.  So it is sufficient
to verify that \eqref{eq:24} holds.

The roots of the quadratic polynomial in $\rho$ on the left hand side
of~\eqref{eq:24} are $\rho_{1,2} = 1+\sin\theta_1\sin\theta_2 \pm
(\sin\theta_1 + \sin\theta_2)$, so that~\eqref{eq:24} holds whenever
$\rho$ is larger or equal to the larger of the two roots, i.e.,
\begin{equation*} \rho \geq 1 + \sin\theta_1\sin\theta_2 +
  \sin\theta_1 + \sin\theta_2 = (1+\sin\theta_1)(1+\sin\theta_2),
\end{equation*} which establishes~\eqref{eq:7}.\qede

\medskip
\noindent This completes the proof of the proposition.

\section{Proof of Corollary~\ref{cor:robust-kl-convergence-rate}.}
\label{app:proof-cor}
We begin with the following lemma.

\begin{lemma}\label{prop sigma}
Let $V\colon \R^2\to \R_+$ be defined as in~\eqref{def V}. Let $\eps \in (0,1)$, and define $\sigma\colon \R_2\to \R_+$,
\begin{align}\label{def sigma}
\sigma(\x)  = \big((1+\eps)^{\frac{1}{2(1+\alpha)}}-1\big) d(\x,A\cap B).
\end{align}
Then for every $\x\in \R^2$,
\begin{align*}%
\sup_{\z \in B[\x,\sigma(\x)]}V(\z) \le (1+\eps)V(\x).
\end{align*}
\end{lemma}

\begin{proof}
Let $\z\in B[\x,\sigma(\x)]$. We have
\begin{align}\label{1st bound rad}
\nonumber \|\x-\z\| & \le \left((1+\eps)^{\frac{1}{2(1+\alpha)}}-1\right)d(\x,A\cap B)
\\
& = \left((1+\eps)^{\frac{1}{2(1+\alpha)}}-1\right) \min\big\{\|\x-\p_1\|, \|\x-\p_2\|\big\},
\end{align}
and so
\begin{eqnarray}\label{z close to p1}
\nonumber \|\z-\p_1\| & \le & \|\z-\x\| + \|\x-\p_1\|
\\
\nonumber & \stackrel{\eqref{1st bound rad}}{\le} & \left((1+\eps)^{\frac{1}{2(1+\alpha)}}-1\right)\|\x-\p_1\| + \|\x-\p_1\|
\\
& = & (1+\eps)^{\frac{1}{2(1+\alpha)}}\|\x-\p_1\|,
\end{eqnarray}
and similarly,
\begin{align}\label{z close to p2}
\|\z-\p_2\| \le (1+\eps)^{\frac{1}{2(1+\alpha)}}\|\x-\p_2\|.
\end{align}
Hence,
\begin{eqnarray*}
V(\z) & = & V_1^\alpha(\z)V_2(\z) = \|\z-\p_1\|^{2\alpha} \|\z-\p_2\|^2 
\\ 
& \stackrel{\eqref{z close to p1}\wedge \eqref{z close to p2}}{\le} & (1+\eps)^{\frac{2\alpha+2}{2(\alpha+1)}}\|\x-\p_1\|^{2\alpha} \|\x-\p_2\|^2
\\
& = & (1+\eps)V(\x).
\end{eqnarray*}
Since $\z\in B[\x,\sigma(\x)]$ is arbitrary, the result follows.
\end{proof}

\begin{proof}[Proof of Corollary~\ref{cor:robust-kl-convergence-rate}]
Let $\x \in \R^2$. By the definition of the $\sigma$-perturbation, we have
\begin{align*}
\z \in (T_{A,B})_\sigma(\x) \iff \z \in \bigcup_{\y \in T_{A,B}(B[\x,\sigma(\x)])}B[\y,\sigma(\y)].
\end{align*}
Therefore, we have
\begin{eqnarray}\label{upper bound V}
\nonumber \sup_{\substack{\z \in (T_{A,B})_\sigma(\x)}}V(\z) & = & \sup_{\substack{\z\in B[\y,\sigma(\y)] \\ \y \in T_{A,B}(B[\x,\sigma(\x)])}}V(\z)
\\
\nonumber & \stackrel{(\clubsuit)}{\le} & (1+\eps)\sup_{\substack{\y  \in T_{A,B}(B[\x,\sigma(\x)])}}V(\y)
\\
\nonumber & \stackrel{(\spadesuit)}{\le} & (1+\eps)\gamma \sup_{\y \in B[\x,\sigma(\x)]}V(\y) 
\\
& \stackrel{(\clubsuit)}{\le} & (1+\eps)^2\gamma V(\x),
\end{eqnarray}
where in ($\clubsuit$) we used Lemma~\ref{prop sigma} and in ($\spadesuit$) we used~\eqref{multi decrease V} in Theorem~\ref{main thm}. 
Let $\x \in \R^2$ and $\phi \in \S_\sigma(\x,T_{A,B})$. Then for all $n \in \Z_+$, we have by~\eqref{upper bound V},
\begin{align}\label{upper bound iter}
V(\phi(\x,n)) \le \left((1+\eps)^2\gamma\right)^nV(\x).
\end{align}

By Theorem~\ref{main thm}, $V$ is a Lyapunov function, and in particular satisfies condition~\eqref{cond omegas}. Therefore, for $\phi \in \S_\sigma(\x,T_{A,B})$,
\begin{eqnarray*}
d(\phi(\x,n),A\cap B)^{2\alpha+2} & \stackrel{\eqref{eq:9}\wedge \eqref{eq:11}}{=} & \varphi_1(\omega_1(\phi(\x,n)))
\\ & \stackrel{\eqref{cond omegas}}{\le} & V(\phi(\x,n)) 
\\ & \stackrel{\eqref{upper bound iter}}{\le} & \left((1+\eps)^2\gamma\right)^nV(\x) 
\\ & \stackrel{\eqref{cond omegas}}{\le} & \left((1+\eps)^2\gamma\right)^n \max\{\|\x-\p_1\|, \|\x-\p_2\|\}^{2\alpha+2}.
\end{eqnarray*}
Altogether, for all $\x\in \R^2$ and $n \in \Z_+$, 
\begin{align*}
\sup_{\phi \in \S_\sigma(\x,T_{A,B})}d(\phi(\x,n),A\cap B) \le 
\max\{\|\x-\p_1\|, \|\x-\p_2\|\}\left((1+\eps)^2\gamma\right)^{\frac{n}{2\alpha+2}}.
\end{align*}
It is trivial to show that the function $\beta\colon \R_+\times \R_+ \to \R_+$ defined by 
\begin{align*}
\beta(s,t) = s\left((1+\eps)^2\gamma\right)^{\frac{t}{2\alpha+2}}
\end{align*}
is a $\kl$-class function. The proof is therefore complete.
\end{proof}

\subsubsection*{Acknowledgements}
\label{sec:acknowledgements}

Ohad Giladi has been supported by ARC grant DP160101537.  Björn S.\
Rüffer has been supported by ARC grant DP160102138.  Both authors
would like to thank the anonymous referees for their helpful comments.

\bibliographystyle{abbrv}
\bibliography{dr-three-lines}

\end{document}